\documentclass[a4paper,12pt]{article}

\usepackage{amsfonts}
\usepackage{amscd,color}
\usepackage{amsmath,amsfonts,amssymb,amscd}
\usepackage{indentfirst,graphicx,epsfig}
\usepackage{graphicx}
\input{epsf}
\usepackage{graphicx}
\usepackage{epstopdf}
\usepackage{caption}

\newtheorem{thm}{Theorem}[section]

\newtheorem{conj}[thm]{Conjecture}
\newtheorem{lem}[thm]{Lemma}

\newtheorem{cor}[thm]{Corollary}

\newenvironment {proof} {\noindent{\em Proof.}}{\hspace*{\fill}$\Box$\par\vspace{4mm}}
\newcommand{\ml}{l\kern-0.55mm\char39\kern-0.3mm}

\title{\textbf{More on rainbow disconnection in graphs\footnote{Supported by NSFC No.11871034, 11531011 and NSFQH No.2017-ZJ-790.}}}
\author{{\small Xuqing Bai, Renying Chang, Xueliang Li } \\
{\small  Center for Combinatorics and LPMC}\\
{\small Nankai University, Tianjin 300071, P.R. China}\\
{\small Email: baixuqing0@163.com, changrysd@163.com, lxl@nankai.edu.cn}\\
}
\date{}
\begin{document}
\maketitle
\begin{abstract}
Let $G$ be a nontrivial edge-colored connected graph. An edge-cut
$R$ of $G$ is called a rainbow cut if no two edges of it are colored
the same. An edge-colored graph $G$ is rainbow disconnected if for
every two vertices $u$ and $v$, there exists a $u-v$ rainbow cut.
For a connected graph $G$, the rainbow disconnection number of $G$,
denoted by $rd(G)$, is defined as the smallest number of colors that
are needed in order to make $G$ rainbow disconnected. In this paper,
we first solve a conjecture that determines the maximum size of a
connected graph $G$ of order $n$ with $rd(G) = k$ for given integers
$k$ and $n$ with $1\leq k\leq n-1$, where $n$ is odd, posed by
Chartrand et al. in \cite{CDHHZ}. Secondly, we discuss bounds of the
rainbow disconnection numbers for complete multipartite graphs,
critical graphs, minimal graphs with respect to chromatic index and
regular graphs, and give the rainbow disconnection numbers for
several special graphs. Finally, we get the Nordhaus-Gaddum-type
theorem for the rainbow disconnection number of graphs. We prove
that if $G$ and $\overline{G}$ are both connected, then $n-2 \leq
rd(G)+rd(\overline{G})\leq 2n-5$ and $n-3\leq rd(G)\cdot
rd(\overline{G})\leq (n-2)(n-3)$. Furthermore, examples are given to
show that the upper bounds are sharp for $n\geq 6$, and the lower
bounds are sharp when $G=\overline{G}=P_4$.


\end{abstract}

\section{Introduction}

All graphs considered in this paper are simple, finite and
undirected. Let $G=(V(G),E(G))$ be a nontrivial connected graph with
the vertex set $V(G)$ and the edge set $E(G)$. For $v\in V(G)$, let
$d_G(v)$ and $N_G(v)$ denote the $degree$ of $v$ and the $neighbour$
of $v$ in $G$, respectively. We use $\delta(G)$ and
$\Delta(G)$ to denote the minimum and maximum degree of $G$. $G_\Delta$ is the
subgraph of $G$ induced by the vertices of maximum degree.
$\overline{G}$ is the $complemet$ of $G$. For any notation or
terminology not defined here, we follow those used in \cite{BM}.

Throughout this paper, we use $P_n$, $C_n$, $K_n$ to denote a path,
a cycle and a complete graph of order $n$, respectively. Given two
disjoint graphs $G$ and $H$, the \emph{join} of two graphs $G$ and
$H$, denoted by $G\vee H$, is obtained from the vertex-disjoint
copies of $G$ and $H$ by adding all edges between $V(G)$ and $V(H)$.

Let $G$ be a graph with an \emph{edge-coloring} $c$:
$E(G)\rightarrow [k] = \{1,2,...,k\}$, $k \in \mathbb{N}$, where
adjacent edges may be colored the same. When adjacent edges of $G$
receive different colors by $c$, the edge-coloring $c$ is called
\emph{proper}. The \emph{chromatic index} of $G$, denoted by $\chi'(G)$,
is the minimum number of colors needed in a proper coloring of $G$.
By a famous theorem of Vizing \cite{V},
$$\Delta(G) \leq \chi'(G) \leq \Delta(G)+1$$
for every nonempty graph $G$. And if $\chi'(G) = \Delta(G)$, then
$G$ is \emph{Class} $1$; if $\chi'(G) = \Delta(G) + 1$, then $G$
is \emph{Class} $2$.

A path is \emph{rainbow} if no two edges of it are colored the same.
An edge-colored graph $G$ is \emph{rainbow~connected} if every two
vertices are connected by a rainbow path. An edge-coloring under
which $G$ is rainbow connected is called a
\emph{rainbow~connection~coloring}. Clearly, if a graph is rainbow
connected, it must be connected. For a connected graph $G$, the
\emph{rainbow connection number} of $G$, denoted by $rc(G)$, is the
smallest number of colors that are needed to make $G$ rainbow
connected. Rainbow connection was introduced by Chartrand et al.
\cite{CJMZ} in $2008$. For more details on rainbow connection, see
the book \cite{LS} and the survey paper \cite{LSS}.

In this paper, we investigate a new concept that is somewhat reverse
to rainbow connection and present some results dealing with this
concept.

An \emph{edge-cut} of a graph $G$ is a set $R$ of edges such that
$G-R$ is disconnected. The minimum number of edges in an edge-cut is
its \emph{edge-connectivity} $\lambda(G)$. We have the well-known
inequality $\lambda(G)\leq \delta(G)$. For two vertices $u$ and $v$,
let $\lambda(u,v)$ denote the minimum number of edges in an edge-cut
$R$ such that $u$ and $v$ lie in different components of $G-R$. The
following result presents an alternate interpretation of
$\lambda(u,v)$ (see \cite{EFS}, \cite{FF}).

\emph{For every two vertices $u$ and $v$ in a graph $G$,
\emph{$\lambda(u,v)$} is the maximum number of pairwise
edge-disjoint $u-v$ paths in $G$}.

The \emph{upper edge-connectivity} $\lambda^+(G)$ is defined by
$\lambda^+(G) = max\{\lambda(u,v): u, v\in V(G)\}$. Consider, for
example, the graph $K_n + v$ obtained from the complete graph $K_n$,
one vertex of which is attached to a single vertex $v$. For this
graph, $\lambda(K_n+v) = 1$ while $\lambda^+(K_n+v) = n-1$. Thus,
$\lambda(G)$ denotes the global minimum edge-connectivity of a
graph, while $\lambda^+(G)$ denotes the local maximum
edge-connectivity of a graph.

An edge-cut $R$ of $G$ is called a rainbow cut if no two edges in
$R$ are colored the same. A rainbow cut $R$ is said to separate two
vertices $u$ and $v$ if $u$ and $v$ belong to different components
of $G-R$. Such rainbow cut is called a $u-v$ rainbow cut. An
edge-colored graph $G$ is \emph{rainbow~disconnected} if for every
two vertices $u$ and $v$ of $G$, there exists a $u-v$ rainbow cut in
$G$. In this case, the edge-coloring $c$ is called a
\emph{rainbow disconnection coloring} of $G$. Similarly, we define
the \emph{rainbow disconnection number} (or {\it RD number} for short)
of $G$, denoted $rd(G)$, as the smallest number of colors that are
needed in order to make $G$ rainbow disconnected. A rainbow disconnection
coloring with $rd(G)$ colors is called an rd-\emph{coloring} of $G$.

A Nordhaus-Gaddum-type result is a (tight) lower or upper bound on the
sum or product of the values of a parameter for a graph and its complement.
The name ¡°Nordhaus-Gaddum-type¡± is given because Nordhaus and Gaddum
\cite{NG} first established the following type of inequalities for chromatic
numbers in $1956$. They proved that if $G$ and $\overline{G}$ are
complementary graphs on $n$ vertices whose chromatic numbers are $\chi(G)$
and $\chi(\overline{G})$, respectively, then
$$2\sqrt{n}\leq \chi(G)+\chi(\overline{G})\leq n+1, n\leq
\chi(G)\cdot\chi(\overline{G})\leq (\frac{n+1}{2})^2.$$
Since then, the Nordhaus-Gaddum type relations have received wide
attention: rainbow connection number \cite{CLL},  Wiener index
\cite{LWYA}, connectivity \cite{HV}, domination number (\cite{HH},
\cite{SDK}), and so on. For more results, we refer to a recent
survey paper \cite{AH}.

The remainder of this paper will be organized as follows. In Section
$2$, we solve a conjecture which was posed by Chartrand et al. in
\cite{CDHHZ}. In Section $3$, we discuss the bounds of rainbow
disconnection numbers of graphs on some parameters and give the
rainbow disconnection numbers of some well-known graphs. In Section
$4$, we get the Nordhaus-Gaddum-type theorem for the rainbow
disconnection number of graphs and prove that the bounds are sharp.
\section{Proof of a conjecture }

In \cite{CDHHZ}, for given integers $k$ and $n$ with $1 \leq k \leq n-1$,
the authors have determined the minimum size of a connected graph $G$ of
order $n$ with $rd(G) = k$. So, this brings up the question of determining
the maximum size of a connected graph $G$ of order $n$ with $rd(G) = k$.
For odd integer $n$, the authors give the following conjecture.
\begin{conj}\label{conj1}
Let $k$ and $n$ be integers with $1 \leq k \leq n-1$ and $n\geq 5$ is odd.
Then the maximum size of a connected graph $G$ of order $n$ with $rd(G)=k$
is $\frac{(k+1)(n-1)}{2}$.
\end{conj}

Before we give the proof of Conjecture \ref{conj1}, some auxiliary lemmas
are stated as follows.
\begin{lem}{\upshape\cite{CDHHZ}}\label{lem1}
If $G$ is a nontrivial connected graph, then
$$\lambda(G) \leq \lambda^+(G)\leq rd(G)\leq \chi'(G) \leq \Delta(G)+1 .$$
\end{lem}

\begin{lem}{\upshape\cite{CDHHZ}}\label{tree}
Let $G$ be a nontrivial connected graph. Then $rd(G)=1$ if and only if
$G$ is a tree.
\end{lem}

\begin{lem}{\upshape\cite{CDHHZ}}\label{cycle}
Let $G$ be a cycle of order $n$. Then $rd(G)=2$.
\end{lem}
\begin{lem}{\upshape\cite{CDHHZ}}\label{kn}
For each integer $n\geq 4$, $rd(K_n)=n-1$.
\end{lem}
\noindent\textbf{Remark 1.} When $n=2$, $G=\{e\}$, then $rd(G)=1$ by Lemma \ref{tree};
when $n=3$, $G=C_3$, then $rd(G)=2$ Lemma \ref{cycle}.
And by Lemma \ref{kn}, we have for any integer $n\geq 2$, $rd(K_n)=n-1$.

\begin{lem}{\upshape\cite{M}}\label{lem2}
Let $G$ be a graph of order $n$ $(n\geq k+2 \geq 3)$. If
$e(G)>\frac{k+1}{2}(n-1)-\frac{1}{2}\sigma_k(G),$ where
$\sigma_k(G)=\sum\limits_{\mbox{\tiny $\begin{array}{c}
             x\in V(G) \\
             d(x)\le k \end{array}$}}(k-d(x))$, then
$\lambda^+(G)\geq k+1.$
\end{lem}

\noindent$Proof ~of ~Conjecture ~\ref{conj1}.$ If $k=n-1$, we have the
maximum size of a connected of order $n$ with $rd(G)=n-1$ is
$\frac{n(n-1)}{2}$ since $rd(K_n)=n-1$ by Remark $1$. Obviously, the result
is true for $k=n-1$.
Now we consider that $1\leq k\leq n-2$. Suppose that
$e(G)>\frac{(k+1)(n-1)}{2}- \frac{1}{2}\sigma_k(G)$, then $rd(G)\geq
\lambda^+(G)\geq k+1$ by Lemma \ref{lem1} and Lemma \ref{lem2}.
Therefore, if $rd(G)=k$, then $e(G) \leq
\frac{(k+1)(n-1)}{2}-\frac{1}{2}\sigma_k(G) \leq
\frac{(k+1)(n-1)}{2}$ since $\sigma_{k}(G)$ is nonnegative.

It remains to show that for each pair $k$, $n$ of integers with
$1\leq k\leq n-2$ and $n\geq 5$ is odd, there exists a connected
graph $G_k$ with order $n$ and size $\frac{(k+1)(n-1)}{2}$ such that
$rd(G_k)=k$.

First, we construct $G_k$ as follows. Set $G_k=H_k \vee K_{1}$,
where $H_k$ is a $(k-1)$-regular graph of order $n-1$ and $K_{1}=\{
u \}$. Since $n-1$ is even, such graph $H_k$ exists. $G_k$ is a
connected graph of order $n$ having one vertex $u$ of degree $n-1$
and $n-1$ vertices of degree $k$, the size of $G_k$ is
$\frac{(k+1)(n-1)}{2}$.

Next, we prove that the rainbow disconnection number of $G_k$ equals
$k$. Since $H_k$ can be selected so that it is $1$-factorable,
$\chi'(H_k)=k-1$. We may obtain a proper $(k-1)$-edge-coloring $c_0$
of $H_k$ using colors from $\{1,2,...,k-1\}$. Extend the edge-coloring
$c_0$ to an edge-coloring $c$ of $G$ by assigning $c(e)=k$ for each
edge $e \in E(G) \backslash E(H_k)$. Under the edge-coloring $c$ of $G$,
the set $E_x$ of edges incident with $x$($x \neq u$) is a rainbow set.
For any two vertices $x$ and $y$ of $G_k$, at least one of $x$ and $y$
is not $u$, say $x \neq u$. We obtain $E_x$ is a $x-y$ rainbow cut,
hence $c$ is a rainbow disconnection coloring of $G$ using $k$
colors. So $rd(G_k)\leq k$.

Furthermore, we show that $rd(G_k)\geq k$. Note that the size of
$G_k$ is $\frac{(k+1)(n-1)}{2}$, and $\frac{(k+1)(n-1)}{2} >
\frac{k(n-1)}{2}\geq \frac{k(n-1)}{2}-\frac{1}{2}\sigma_{k-1}(G_k)$
since $\sigma_{k-1}(G_k)$ is nonnegative. Thus, $\lambda^+(G_k)\geq
k$ by Lemma \ref{lem2}. Combining with Lemma \ref{lem1}, we have
$rd(G_k)\geq k$. ~~~~~$\Box$

\section{The RD numbers of some classes of graphs }

In this section, we discuss the rainbow disconnection numbers of
complete multipartite graphs, critical graphs, minimal graphs with
respect to chromatic index and regular graphs.

First, we give the rainbow disconnection numbers of complete
multipartite graphs.
\begin{thm}\label{thm1}
Let $G=K_{n_1,n_2,...,n_k}$ be a complete $k$-partite graph with
order $n$ where $k\geq 2$ and $n_1\leq n_2\leq \cdots \leq n_k$.
Then
$$rd(K_{n_1,n_2,...,n_k})=
\begin{cases}
n-n_2& \text{if $n_1=1$},\\
n-n_1& \text{if $n_1\geq 2$}.
\end{cases}$$
\end{thm}

For the proof of Theorem \ref{thm1} we give two lemmas as follows.
\begin{lem}{\upshape\cite{CDHHZ}}\label{lem4}
Let $H$ be a connected subgraph of a graph $G$. Then $rd(H)\leq rd(G)$.
\end{lem}

\begin{lem}{\upshape\cite{CDHHZ}}\label{lem3}
Let $G$ be a nontrivial connected graph of order $n$. Then $rd(G)=n-1$
if and only if $G$ contains at least two vertices of degree $n-1$.
\end{lem}

\noindent$Proof~ of~ Theorem~ \ref{thm1}.$ Let $V_1,V_2,\ldots V_k$
be the partite vertex sets of $G$ with
$V_i=\{v_{i,1},v_{i,2},\ldots, v_{i,n_i}\}$ where $1\leq i\leq k$.

\textbf{Case 1.}
$n_1=n_2=1$.

In this case, we have $V_1=\{v_{1,1}\}$, $V_2=\{v_{2,1}\}$ and
$d(v_{1,1})=d(v_{2,1})=n-1$.
Then the graph $G$ has at least two vertices of degree $n-1$,
so $rd(G)=n-1$ by Lemma \ref{lem3}.

\textbf{Case 2.}
$n_1=1$ and $n_2\geq 2$.

First, we have $V_1=\{v_{1,1}\}$ and $d(v_{1,1})=n-1$. Let
$H=G-\{v_{1,1}\}$. Then $\Delta(H) = n-n_2-1$. Since $\chi'(H) \leq
\Delta(H)+1\leq n-n_2$ by Vizing theorem \cite{V}, there is a proper
edge-coloring $c_0$ of $H$ using $n-n_2$ colors. For each vertex
$x\in V(H)$, $d_H(x)\leq n-n_2-1$, at least one of the $n-n_2$
colors is missing from the colors of the edges incident with $x$ in
$H$. Let $a_{x}$ be one such missing color. Since $E(G)=E(H)\cup
\{v_{1,1}x \mid x\in V(H)\}$, we now extend $c_0$ to an
edge-coloring $c$ of $G$ by assigning $c(v_{1,1}x)=a_x$ for each
vertex $x\in V(H)$. Note that the set $E_x$ of edges incident
with $x$ is a rainbow set for each vertex $x\in V(H)$. Let $u$ and
$v$ be two vertices of $G$. Then at least one of $u$ and $v$ belongs
to $H$, say $u\in V(H)$. $E_u$ is a $u-v$ rainbow cut, it follows
that $c$ is a rainbow disconnection coloring of $G$ using $n-n_2$
colors. Therefore, $rd(G)\leq n-n_2$.

For any two vertices $u$, $v$ of $V_2$, they are adjacent with all
the vertices of $V(G)\setminus V_2$, that is, $\lambda(u,v) \geq
n-n_2$. It follows by Lemma \ref{lem1} that $rd(G) \geq n-n_2$.
Hence, $rd(G)=n-n_2$.

\textbf{Case 3.}
$n_1\geq 2$.

\textbf{Case 3.1.} The number of vertices of $k$ partite set is not
completely equal. First, let $i$ be the minimum value such that $n_i
\neq n_1$. We have $n_i\geq n_1+1$ where $i\geq 2$. Let $u$ be a vertex
of $V_i$ and $F=G-u$. Then $\Delta(F) = n-n_1-1$. Since $\chi'(F)
\leq \Delta(F)+1\leq n-n_1$ by Vizing theorem \cite{V}, there is a
proper edge-coloring $c_0$ of $F$ using $n-n_1$ colors. For each
vertex $x\in V(F)$, $d_F(x)\leq n-n_1-1$, similarly, there is a
$a_x\in [n-n_1]$ such that $a_x$ is not assigned to any edge
incident with $x$ in $F$. Since $E(G)=E(F)\cup \{ux\mid x\in
N_G(u)\}$, we now extend the edge-coloring $c_0$ of $F$ to an
edge-coloring $c$ of $G$ by assigning $c(ux)=a_x$ for any vertex
$x\in N_G(u)$. Likewise, the set $E_x$ of edges incident with $x$ is
rainbow for each vertex $x\in V(F)$. Let $v$ and $w$ be two vertices
of $G$. Then at least one of $v$ and $w$ belongs to $F$, say $v\in
V(F)$. Since $E_v$ is a $v-w$ rainbow cut, $c$ is a rainbow
disconnection coloring of $G$ using $n-n_1$ colors. Therefore,
$rd(G)\leq n-n_1$.

For any two vertices of $V_1$, all vertices of $V(G)\setminus V_1$
are their common neighbours. Then $\lambda^{+}(G)\geq n-n_1$, it
follows by Lemma \ref{lem1} that $rd(G) \geq n-n_1$. Hence,
$rd(G)=n-n_1$.

\textbf{Case 3.2.} The number of vertices of $k$ partite set is
equal. That is $n_1=n_2=\cdots=n_k\geq 2$. Now, we construct a graph
$G^{*}=K_{n_1,n_2,...,n_k+1}$ be a complete $k$-partite graph. Then
it follows by Case $3.1$ that $rd(G^{*})=n-n_1$. Furthermore, since
$G$ is a subgraph of $G^{*}$, $rd(G)\leq rd(G^{*})\leq n-n_1$  by
Lemma \ref{lem4}. Similarly, for any two distinct vertices of $V_1$,
all vertices of $V(G)\setminus V_1$ are their common neighbours. Then
$\lambda^+(G)\geq n-n_1$, it follows by Lemma \ref{lem1} that $
rd(G)\geq n-n_1$. Hence,
$rd(G)=n-n_1$.~~~~~~~~~~~~~~~~~~~~~~~~~~~~~~~~~~~~~~~~~~~~~~~~~~~~~~~~~~~~~~
~~~~~~~~~~~~~~~~~~~~~~~~~~~~~~~~~~~~~~~~~~~~~~~~~~~~~~~~~~~~~~$\Box$

A graph $G$ is said to be \emph{colour-critical} if
$\chi(H)<\chi(G)$ for every proper subgraph $H$ of $G$. The study of
critical $k$-chromatic graphs was started by Dirac (\cite{D1},
\cite{D2}). Here, for simplicity, we abbreviate the term
``color-critical" to ``critical". A $k$-\emph{critical} graph is one
that is $k$-chromatic and critical. We get a lower bound of rainbow
disconnection number for $(k+1)$-critical graph.

\begin{thm}\label{thm2}
Let $G$ be a connected $(k+1)$-critical graph. Then $rd(G)\geq k$.
\end{thm}

We proceed our proof by the following two lemmas. First,
we give a lower bound of rainbow disconnection number on average
degree of $G$.
\begin{lem}\label{deg}
Let $G$ be a connected graph of order $n$ with average degree $d$.
Then $rd(G)\geq \lfloor d\rfloor$.
\end{lem}

\begin{proof}
If $G$ is a tree, then $1\leq d<2$ since $d=\frac{2(n-1)}{n}$. And
we have $rd(G)=1$ by Lemma \ref{tree}. Obviously $rd(G)=1\geq\lfloor
d\rfloor$ and the result is true. If $G$ is not a tree, then $d\geq 2$
since $\frac{2e(G)}{n}\geq \frac{2n}{n}=2$. We have $e(G)=\frac{1}{2}dn
\geq\frac{1}{2}\lfloor d\rfloor n>\frac{1}{2}\lfloor d\rfloor(n-1)$,
so $\lambda^+(G)\geq \lfloor d\rfloor$ by Lemma \ref{lem2}.
Therefore, $rd(G)\geq \lfloor d\rfloor$ by Lemma \ref{lem1}.
\end{proof}

\begin{lem}{\upshape\cite{D1}}\label{cri}
Let $G$ be a connected $(k+1)$-critical graph. Then $\delta(G)\geq k$.
\end{lem}

\noindent$Proof~ of~ Theorem ~\ref{thm2}.$ Let $G$ be a
$(k+1)$-critical with the average degree $d$. We know $\delta(G)\geq k$
by Lemma \ref{cri}. Obviously, $d\geq \delta(G)\geq k$. Therefore, it follows by
Lemma \ref{deg} that $rd(G)\geq \lfloor d\rfloor \geq k$ since $k$ is an integer.
$~~~~~~~~~~~~~~~~~~~~~~~~~~~~~~~~~~~~~~~~~~~~~~~~~\Box$

A graph $G$ with at least two edges is \emph{minimal} with respect
to chromatic index if $\chi'(G-e) = \chi'(G)-1$ for any edge $e$ of
$G$. We show that the rainbow disconnection number of connected
minimal graphs with respect to chromatic index is less than maximum
degree.

\begin{thm}\label{mini}
Let $G$ be a connected minimal graph with respect to chromatic
index. Then $rd(G)\leq \Delta(G)$.
\end{thm}

The following lemma will be used for the proof of Theorem \ref{mini}.

\begin{lem}{\upshape\cite{B}}\label{min}
Let $G$ be a connected graph with $\Delta(G)=d\geq 2$. Then $G$ is
minimal with respect to chromatic index if and only if either:

i) $G$ is Class $1$ and $G = K_{1,d}$ or

ii) $G$ is Class $2$ and $G-e$ is Class $1$ for every edge $e$ of $G$.
\end{lem}

\noindent$Proof~ of~ Theorem ~\ref{mini}.$  Let $G$ be a minimal
connected graph with respect to chromatic index. We distinguish the
following two cases according to Lemma $\ref{min}$.

\textbf{Case 1.} $G$ is Class $1$ and $G=K_{1,d}$ with $d\geq2$. It
follows that $rd(G)=1$ by Lemma \ref{tree}, obviously $rd(G)<
d=\Delta(G)$.

\textbf{Case 2.} $G$ is Class $2$ and for any edge $e \in E(G)$,
$\chi'(G-e)=\Delta(G-e)$. We pick one vertex $v \in V(G)$ such that
$d_{G}(v)= \Delta(G)$. Let $H=G-uv$ for some vertex $u \in N_{G}(v)$.
Then $\chi'(H)=\Delta(H)$ and $\chi'(H)=\chi'(G)-1=\Delta(G)$
since $G$ is minimal with respect to chromatic index and $G$ is Class $2$.
Thus, it implies that $\chi'(H)=\Delta(H)=\Delta(G)$. First we obtain a
proper edge-coloring $c_0$ of $H$ using colors from
$[\Delta(G)]=\{1,2,\cdots,\Delta(G)\}$.
Since $d_{H}(v) <
\Delta(G)$, there is a $a_v\in[\Delta(G)]$ such that $a_v$ is not
assigned to any edge incident with $v$ in $H$. Now we extend $c_0$
to an edge-coloring $c$ of $G$ by defining $c(uv)=a_v$. Note that
the set $E_x$ of edges incident with $x$ in $G$ is a rainbow set for
each vertex $x\in V(G)\setminus u$ in both cases. Let $p$ and $q$
be two vertices of $G$. Then at least one of $p$ and $q$ is not $u$,
say $p\neq u$. Since $E_p$ is a $p-q$ rainbow cut, $c$ is a rainbow
disconnection coloring of $G$ using at most $\Delta(G)$ colors.
Therefore, $rd(G)\leq \Delta(G)$.
$~~~~~~~~~~~~~~~~~~~~~~~~~~~~~~~~~~~~~~~~~~~~~~~~~~~~~~~~~~~~~~~~
~~~~~~~~~~~~~~~~~~~~~~~~~~\Box$

For regular graphs, we know that not all $k$-regular graph have
$rd(G)=k$. For example, we know that the Petersen graph $P$ is a
$3$-regular graph but $rd(P)=4$ in \cite{CDHHZ}. The following
theorems give some regular graphs satisfying $rd(G)=k$.
\begin{thm}\label{reg2}
Let $G$ be a connected $k$-regular graph of even order satisfying
$k\geq \frac{6}{7}|V(G)|$. Then $rd(G)=k$.
\end{thm}

\begin{thm}\label{reg3}
Let $G$ be a connected $k$-regular bipartite graph. Then $rd(G)=k$.
\end{thm}

\begin{thm}\label{reg4}
Let $G$ be a connected $(n-k)$-regular graph with order $n$, where
$1 \leq k\leq 4$. Then $rd(G)=n-k$.
\end{thm}

Here, we list the following several lemmas, which will be used in
this work.

\begin{lem}{\upshape\cite{ACCGZ}}\label{classs}
Let $G$ be a connected graph. If every connected component of $G_\Delta$
is a unicyclic graph or a tree, and $G_\Delta$ is not a disjoint union of
cycles, then $G$ is Class $1$.
\end{lem}

\begin{lem}{\upshape\cite{CH}}\label{even}
Let $G$ be a regular graph of even order $n$ and degree $d(G)$ equal to
$n-3$, $n-4$, or $n-5$. Let $d(G)\geq 2\lfloor\frac{1}{2}(\frac{n}{2}+1)\rfloor-1$.
Then $G$ is Class $1$.
\end{lem}

\begin{lem}{\upshape\cite{CH}}\label{lem5}
Let $G$ be a regular graph of even order and degree $d(G)$
satisfying $d(G)\geq \frac{6}{7}|V(G)|$. Then $G$ is Class $1$.
\end{lem}

For regular graphs, it is easy to get the following result.
\begin{lem}\label{regu}
Let $G$ be a connected $k$-regular graph. Then $k \leq rd(G)\leq
k+1$.
\end{lem}

\begin{proof}
Since average degree of $k$-regular graph $G$ is $k$, it
follows by Lemma \ref{deg} that $rd(G)\geq k$ .
Furthermore, $rd(G)\leq \chi'(G)\leq \Delta+1=k+1$ by Lemma \ref{lem1}.
\end{proof}

\noindent$ Proof~ of~ Theorem~ \ref{reg2}$. Let $G$ be a connected
$k$-regular graph of even order satisfying $k\geq
\frac{6}{7}|V(G)|$. We have $G$ is Class $1$ by Lemma \ref{lem5}.
Thus $\chi'(G)=k$. And as the above argument and Lemma \ref{lem1},
we get $rd(G)=k$.
$~~~~~~~~~~~~~~~~~~~~~~~~~~~~~~~~~~~~~~~~~~~~\Box$

$ $

\noindent$ Proof~ of~ Theorem~  \ref{reg3}.$ Since $G$ is
a bipartite graph, $\chi'(G)=\Delta(G)=k$ (see \cite{BM}). And by Lemma
\ref{lem1} and Lemma \ref{regu}, we
have $rd(G) = k$.
$~~~~~~~~~~~~~~~~~~~~~~~~~~~~~~~~~~~~~~~~~~~~~~~~~\Box$

$ $

\noindent$ Proof~ of~ Theorem~ \ref{reg4}$. We distinguish the
following three cases.

\textbf{Case 1.} $k=1$. We have $G=K_n$, it is true for $n\geq 2$
by Remark $1$.

\textbf{Case 2.} $k=2$ or $3$. Let $u\in V(G)$ and consider the
graph $H=G-u$. Then $\Delta(H)=n-k$ and the number of maximum degree
vertices of $H$ is one or two. So each component of $H_\Delta$ is a
tree. Therefore, it follows by Lemma \ref{classs} that $H$ is Class
$1$, that is $\chi'(H)=n-k$. We now obtain a proper
edge-coloring $c_{0}$ of $H$ using colors from $[n-k]$. For each
vertex $x\in N_G(u)$, $d_H(x)\leq n-k-1$, there is a $a_x\in [n-k]$
such that $a_x$ is not assigned to any edge incident with $x$ in
$H$. Since $E(G)=E(H)\cup \{ux\mid x\in N_G(u)\}$, we now extend the
edge-coloring $c_0$ of $H$ to an edge-coloring $c$ of $G$ by
assigning $c(ux)=a_x$ for any vertex $x\in N_G(u)$. Note that the set
$E_x$ of edges incident with $x$ is a rainbow set for each vertex $x\in
V(H)$. Let $v$ and $w$ be two vertices of $G$. Then at least one of
$v$ and $w$ belongs to $H$, say $v\in V(H)$. Since $E_v$ is a $v-w$
rainbow cut, $c$ is a rainbow disconnection coloring of $G$ using
$n-k$ colors. Therefore, $rd(G)\leq n-k$. By Lemma $\ref{regu}$,
$rd(G)\geq n-k$. Thus, $rd(G) = n-k$.

\textbf{Case 3.} $k=4$. Let $G$ be a $(n-4)$-regular graph with
order $n$, where $n\geq 5$. Then we know the $n$ must be even since
$2m=n(n-4)$. First, we consider $n\geq 8$. It is easy to verify that
$d(G)=n-4\geq 2\lfloor\frac{1}{2}(\frac{n}{2}+1)\rfloor-1$, it
follows by Lemma \ref{even} that $G$ is Class $1$. So,
$\chi'(G)=n-4$. Furthermore, we get $rd(G)=n-4$ by Lemma \ref{lem1}
and Lemma \ref{regu}. Secondly, it remains to consider case for
$n=6$ since $n$ is even. In this case, we have $G=C_6$. By Lemma
\ref{cycle}, we obtain $rd(G)=2=n-4$.~~~~~~~~~~~~~~~$\Box$

\section{Nordhaus-Gaddum-type results}

In the sequel, we study Nordhaus-Gaddum-type problem for rainbow
disconnection number of graphs $G$. We know that if $G$ is a
connected graph with $n$ vertices, the number of the edges in $G$ is
at least $n-1$. Since $2(n-1)\leq
e(G)+e(\overline{G})=e(K_n)=\frac{n(n-1)}{2}$, if both $G$ and
$\overline{G}$ are connected, $n$ is at least $4$.

In the rest of the paper, we always assume that all graphs have at
least $4$ vertices, both $G$ and $\overline{G}$ are connected. For
any vertex $u\in V(G)$, let $\bar{u}$ denote the vertex in
$\overline{G}$ corresponding to the vertex $u$. Now we give a
Nordhaus-Gaddum-type result for rainbow disconnection number.
\begin{thm}\label{ng}
Let $G$ and $\overline{G}$ be connected graph of order $n$.
Then $n-2\leq rd(G)+rd(\overline{G})\leq 2n-5$ and $n-3\leq rd(G)\cdot
rd(\overline{G})\leq (n-2)(n-3)$.
Furthermore, the bounds are sharp.
\end{thm}

For the proof of Theorem \ref{ng} we show some preparatory results
as follows.
\begin{lem}{\upshape\cite{CDHHZ}}\label{block}
Let $G$ be a connected graph, and let $B$ be a block of $G$ such
that $rd(B)$ is maximum among all blocks of $G$. Then $rd(G)=rd(B)$.
\end{lem}

\begin{lem}\label{21}
Let $G$ be a connected graph of order $n$.
If $G$ has at least two vertices of degree $1$, then $rd(G)\leq n-3$.
\end{lem}

\begin{proof}
Let $B$ be a block of $G$ such that $rd(B)$ is maximum among all blocks
of $G$. Then $|V(B)|\leq n-2$ since $G$ has at least two vertices of
degree $1$. It follows by Lemma \ref{lem4} and Remark $1$ that
$rd(B)\leq rd(K_{n-2})=n-3$. And by Lemma \ref{block},
$rd(G)=rd(B)\leq n-3$.
\end{proof}

\begin{lem}\label{n-3}
Let $G$ be a connected graph of order $n$ and contain at most one vertex
of degree at least $n-2$. Then $rd(G)\leq n-3$.
\end{lem}

\begin{proof}
We distinguish three cases.

\textbf{Case 1.} There exists exactly one vertex, says $u$, of
degree $n-1$. Let $F=G-u$. We have $\Delta(F)\leq n-4$ since
$d_G(u)=n-1$ and $d_G(v)\leq n-3$ for any vertex $v\in V(G)\setminus u$.
Therefore, $\chi'(F)\leq n-3$. And we may obtain a proper
edge-coloring of $F$ using colors from $[n-3]$. For each vertex $x
\in N_G(u)$, since $d_F(x)\leq n-4$, there is a $a_x\in [n-3]$ such
that $a_x$ is not assigned to any edge incident with $x$ in $F$.
Since $E(G)=E(F)\cup \{ux\mid x\in N_G(u)\}$, we now extend the
edge-coloring $c_0$ of $F$ to an edge-coloring $c$ of $G$ by
assigning $c(ux)=a_x$ for any vertex $x\in N_G(u)$. Note that the set
$E_x$ of edges incident with $x$ is a rainbow set for each $x\in
V(F)$. Let $v$ and $w$ be two vertices of $G$. Then at least one of
$v$ and $w$ belongs to $F$, say $v\in V(F)$. Since $E_v$ is a $v-w$
rainbow cut, $c$ is a rainbow disconnection coloring of $G$ using
$n-3$ colors. Therefore, $rd(G)\leq n-3$.

\textbf{Case 2.}
There exists exactly one vertex, says $u$, of degree $n-2$.

Let $F=G-u$. If $\Delta(F)\leq n-4$, then $\chi'(F)\leq n-3$. As we
discussed in Case 1, we may obtain a rainbow disconnection coloring
of $G$ using $n-3$ colors. Otherwise, if $\Delta(F)=n-3$, then there
exists exactly one vertex, says $v$, with degree $n-3$ in $F$. We
claim that $F$ is Class $1$. Since $v$ is only one vertex with
$d_{F}(v)=\Delta(F)=n-3$, that is $F_\Delta$ is a tree (single
vertex), it follows by Lemma \ref{classs} that $F$ is Class $1$. So
$\chi'(F)=\Delta(F)=n-3$. We may get a proper edge-coloring $c_{0}$
of $F$ using colors from $[n-3]$. Since $v \notin N_{G}(u)$, for
each vertex $x \in N_{G}(u)$, $d_F(x)\leq n-4$, there is a $a_x\in
[n-3]$ such that $a_x$ is not assigned to any edge incident with $x$
in $F$. And $E(G)=E(F)\cup \{ux\mid x\in N_G(u)\}$, we extend $c_{0}$
to an edge-coloring $c$ of $G$ by setting $c(ux)=a_{x}$. Likewise,
$c$ is a rainbow disconnection coloring of $G$ using $n-3$ colors.
Therefore, $rd(G)\leq n-3$.

\textbf{Case 3.}
$\Delta(G)\leq n-3$.

If $\Delta(G)\leq n-4$, then $rd(G)\leq \chi'(G)\leq n-3$ by Lemma
\ref{lem1}. Thus, we may assume that $\Delta(G)=n-3$. Let $d(u)=n-3$
and $F=G-u$. If $\Delta(F) \leq n-4$, then $\chi'(F)\leq n-3$ by
Lemma \ref{lem1}. Then we obtain a rainbow disconnection coloring of
$G$ using $n-3$ colors as same as Case 1. If $\Delta(F)=n-3$, then
there exist at most two vertices of degree $n-3$ in $F$. So the each
component of $F_\Delta$ is a tree. It follows by Lemma \ref{classs}
that $F$ is Class $1$. Then $\chi'(F)=\Delta(F)=n-3$. We may get a
proper edge-coloring $c_{0}$ of $F$ using colors from $[n-3]$. Since
$\Delta(G)\leq n-3$, for each vertex $x \in N_{G}(u)$, we have
$d_F(x)\leq n-4$. Hence there is a $a_x\in [n-3]$ such that $a_x$ is
not assigned to any edge incident with $x$ in $F$. And
$E(G)=E(F)\cup \{ux\mid x\in N_G(u)\}$, we extend $c_{0}$ to an
edge-coloring $c$ of $G$ by assigning $c(ux)=a_{x}$. As the above
argument, $c$ is a rainbow disconnection coloring of $G$ using $n-3$
colors. Therefore, $rd(G)\leq n-3$.

\end{proof}
By the above Lemma \ref{n-3}, we can immediately get the following result.

\begin{cor}\label{n-2}
Let $G$ be a connected graph with order $n$. If $rd(G)\geq n-2$,
then there are at least two vertices of degree at least $n-2$.
\end{cor}

Now we are ready to prove Theorem \ref{ng}.

\noindent$Proof~ of ~Theorem~ \ref{ng}$.
Let $d$ and $\bar{d}$ be the average degree of $G$ and $\overline{G}$
respectively. Then $rd(G)\geq \lfloor d\rfloor$ and $rd(\overline{G})
\geq \lfloor\bar{d}\rfloor$
by Lemma \ref{deg}. Thus,
\begin{eqnarray*}
rd(G)+rd(\overline{G}) &\geq& \lfloor d\rfloor+ \lfloor\bar{d}\rfloor\\
&\geq& \lfloor d+\bar{d}\rfloor -1\\
&=& \lfloor \frac{2e(G)}{n}+\frac{2e(\overline{G})}{n} \rfloor -1\\
&=& \lfloor \frac{2}{n}\cdot \frac{n(n-1)}{2}\rfloor -1\\
&=& n-2.
\end{eqnarray*}
And the minimum value of $rd(G)\cdot rd(\overline{G})$ is achieved when $rd(G)=1$
and $rd(\overline{G})=n-3$ or $rd(\overline{G})=1$ and $rd(G)=n-3$.
Furthermore, Since both $G$ and $\overline{G}$ are connected, it follows
that $G$ and $\overline{G}$ have $\Delta(G), \Delta(\overline{G})\leq n-2$.
Thus, $rd(G), rd(\overline{G})\leq n-2$ by Lemma \ref{lem3}. Therefore, $n-2\leq
rd(G)+rd(\overline{G})\leq 2n-4$ and $n-3\leq rd(G)\cdot
rd(\overline{G})\leq (n-2)^2$. Now, we prove the two upper bounds
are not true. Assume that $rd(G)+rd(\overline{G})=2n-4$ or
$rd(G)\cdot rd(\overline{G})=(n-2)^2$, that is
$rd(G)=rd(\overline{G})=n-2$. Then $G$ has at least two vertices of
degree $n-2$ since $rd(G)=n-2$ by Corollary \ref{n-2}. Then we get
$\overline{G}$ has at least two vertices of degree $1$. It follows
by Lemma \ref{21} that $rd(\overline{G})\leq n-3$, this is a
contradiction with $rd(\overline{G})=n-2$. Hence, we have $n-2\leq
rd(G)+rd(\overline{G})\leq 2n-5$ and $n-3\leq rd(G)\cdot
rd(\overline{G})\leq (n-2)(n-3)$.

Furthermore, we prove that the bounds are both sharp. First, for the
lower bound, $G=P_4$ is a graph satisfying $rd(G)+rd(\overline{G})=n-2$ and $
rd(G)\cdot rd(\overline{G})=n-3$. Let $G=P_4$. Then $\overline{G}=P_4$.
So $rd(G)=rd(\overline{G})=1$ by Lemma \ref{tree}. Therefore, $
rd(G)+rd(\overline{G})=2$ and $ rd(G)\cdot rd(\overline{G})=1$.

Second, for the upper bound, we construct a graph $G$ of order $n$,
where $n\geq 6$, that satisfying $rd(G)+rd(\overline{G})=2n-5$ and
$rd(G)\cdot rd(\overline{G})=(n-2)(n-3)$ as follows. Let $G$ be a
graph of order $n$ and $u, v, w \in V(G)$. We join the edges $uv$
and $xu$, $xv$ for all $x\in V(G)\setminus \{u, v, w\}$, and then
add an edge $wy$, where $y$ is one vertex of $V(G)\setminus \{u,
v, w\}$. Obviously, the $G$ and $\overline{G}$ be both connected graph.
Now we claim that $rd(G)+rd(\overline{G})=2n-5$ and $rd(G)\cdot
rd(\overline{G})=(n-2)(n-3)$. We only need to prove that
$rd(G)+rd(\overline{G})\geq 2n-5$ and $rd(G)\cdot
rd(\overline{G})\geq (n-2)(n-3)$. First, we have $\lambda(u,v) =
n-2$ by construction of $G$, so $rd(G)\geq n-2$ by Lemma \ref{lem1}.
Next, for any two vertices $p, q \in
V(\overline{G})\setminus\{\bar{u}, \bar{v}, \bar{w}, \bar{y}\}$, we
have $\lambda(p,q)= n-3$ since $x$ is common neighbour of $p$ and
$q$ for each vertex $x\in V(G)\setminus \{ \bar{u}, \bar{v}, p, q
\}$ and $pq$ is an edge in $\overline{G}$. So $rd(\overline{G})\geq
n-3$ by Lemma \ref{lem1}. Hence, $rd(G)+rd(\overline{G})\geq 2n-5$
and $rd(G)\cdot rd(\overline{G})\geq (n-2)(n-3)$.
$~~~~~~~~~~~~~~~~~~~~~~~~~~~~~~~~~~~~~~~~~~~~~~~~~
~~~~~~~~~~~~~~~~~~~~~~~~~~~~~~~~~~~~~~~\Box$

\end{document}